\documentclass[amstex,10pt,reqno]{amsart}
\usepackage{amsmath,amsfonts,amssymb,amsthm,hyperref,enumerate,multicol,graphicx,cite}
\newtheorem{lemma}{Lemma}
\newtheorem{thm}{Theorem}

\newtheorem{example}{Example}

\newtheorem{remark}{Remark}
\numberwithin{equation}{section}
\begin{document}

\leftline{ \scriptsize \it}
\title[]
{Generalized Bernstein-Durrmeyer operators of blending type}
\maketitle

\begin{abstract}
In this article we present the Durrmeyer variant of generalized Bernstein  operators that preserve the constant functions involving non-negative parameter $\rho$. We derive the approximation behaviour of these operators including global approximation theorem via Ditzian-Totik modulus of continuity, the order of convergence for the Lipschitz type space. Furthermore, we study a Voronovskaja type asymptotic formula and local approximation theorem by means of second order modulus of smoothness. Furthermore, we obtain the rate of approximation for absolutely continuous functions having a derivative equivalent with a function of bounded variation. In the last section of the article, we illustrate the convergence of these operators for certain functions using Maple software. 
\end{abstract}

\section{Introduction}
 Bernstein introduced the most famous algebraic polynomials ${B}_n(f;x)$ in approximation theory in order to give a constructive proof of Weierstrass's theorem which is given by
$${B}_n(f;x)=\displaystyle\sum_{k=0}^np_{n,k}(x)f\left(\frac{k}{n}\right), \,\,\ x\in[0,1],$$
where $p_{n,k}(x)=\displaystyle{n\choose k}x^k(1-x)^{n-k}$ and he proved that if $f\in C[0,1]$ then $B_n(f;x)$ converges uniformly to $f(x)$ in $[0,1].$ \\
The Bernstein operators have been used in many branches of mathematics and computer science. Since their useful structure, Bernstein polynomials and their generalizations have been intensively studied. Among the other we refer the readers to (cf. \cite{TA,  GG, CGR, MAK, MYZ, HIR, TBE}).\\

For $f\in C(J)$ with $J=[0,1],$  Chen et al. \cite{CTLX} introduced a vital generalization of the Bernstein operators depending a non-negative real parameter $\alpha$ $(0\leq\alpha\leq1)$ as
\begin{eqnarray}\label{ab1}
T_n^{(\alpha)}(f;x)=\sum_{k=0}^{n}p_{n,k}^{(\alpha)}(x)f\left(\frac{k}{n}\right), \quad x\in J
\end{eqnarray}
where
$p_{n,k}^{(\alpha)}(x)=\displaystyle\left[{n-2\choose k}(1-\alpha)x+{n-2\choose k-2}(1-\alpha)(1-x)+{n\choose k}\alpha x(1-x)\right]x^{k-1}(1-x)^{n-k-1}$ and $n\geq2.$ They obtained a Voronovskaja type asymptotic
formula, the rate of approximation in terms of the modulus of smoothness, and Shape preserving properties for these operators. In the particular case, $\alpha=1,$ these operators reduce the well-known Bernstein operators. Kajla and Acar \cite{E4} introduced the Durrmeyer variant of the operators (\ref{ab1}) and investigated the rate of approximation of these operators.  \\

Gonska and P\v{a}lt\v{a}nea \cite{GP} presented genuine Bernstein-Durrmeyer type operators and obtained the simultaneous approximation for these operators. Gupta and Rassias \cite{GR} studied approximation behavior of Durrmeyer type of Lupa\c{s} operators based on Polya distribution Goyal et al. \cite{GVP} derived Baskakov-Sz\`{a}sz type operators and studied quantitative convergence theorems for these operators. Gupta et al. \cite{VAS} introduced a hybrid operators based on   inverse Polya-Eggenberger distribution and studied degree of approximation and uniform convergence. Acu and Gupta \cite{AGG} introduced a summation-integral type operators involving two parameters and studied some direct results e.g. Voronovskaja type asymptotic formula, local approximation and weighted approximation of these operators.  Very recently, Kajla and Goyal \cite{KG} considered the hybrid operators involving non-negative parameters and investigated their order of approximation. In the literature survey, several researchers  has been studied  the approximation properties of hybrid operators [cf. \cite{ AGII, TA3, TA4, AGA,  AK, VRP, KAA, AAC, KLA, RPA, TAI, DCM}].\\

For $f\in C(J),$ we construct the following Durrmeyer variant of the operators (\ref{ab1}) depending on parameter $\rho>0$ as follows:
\begin{eqnarray}\label{ma2}
\mathcal{G}_{n,\rho}^{(\alpha)}(f;x)=\sum_{k=0}^{n}p_{n,k}^{(\alpha)}(x) \int_{0}^1 \mu_{n,\rho}(t)f(t) dt,
\end{eqnarray}
where $\mu_{n,\rho}(t)=\dfrac{t^{k\rho}(1-t)^{(n-k)\rho}}{B\left(k\rho+1,(n-k)\rho+1\right)}$ and $B\left(k\rho+1,(n-k)\rho+1\right)$ is the beta function defined by $B(e,f)=\displaystyle\int_0^1t^{e-1}(1-t)^{f-1}dt=\dfrac{\Gamma (e) \Gamma (f)}{\Gamma(e+f)},$ $e,f>0$ and $p_{n,k}^{(\alpha)}(x)$ is defined as above. It is seen that the operators $\mathcal{G}_{n,\rho}^{(\alpha)}$ reproduce the constant functions.

The aim of this note is to find the approximation properties for the generalized Bernstein-Durrmeyer operators involving a nonnegative parameter of the operators defined in (\ref{ma2}). We give a Voronovskaja type theorem, global approximation theorem by means of Ditzian-Totik modulus of smoothness, Lipschitz type space and a local approximation theorem with the help of second order modulus of continuity. Furthermore, we study the rate of approximation for absolutely continuous functions having a derivative equivalent with a function of bounded variation. Lastly, we illustrate the convergence of these operators for certain functions using Maple software.

\section{Auxiliary Results }
In order to prove main results, we will show some lemmas in this section.
Let $e_i(x)=x^i, i=\overline{0,4}$
\begin{lemma}\label{l1}
For the generalized Bernstein-Durrmeyer operators $\mathcal{G}_{n,\rho}^{(\alpha)}(f;x),$ we have
\begin{enumerate}[(i)]
\item $\mathcal{G}_{n,\rho}^{(\alpha)}(e_0;x)=1;$
\item $\mathcal{G}_{n,\rho}^{(\alpha)}(e_1;x)=\dfrac{n\rho x+1}{n\rho+2};$
\item $\mathcal{G}_{n,\rho}^{(\alpha)}(e_2;x)=\dfrac{x^2\rho^2\left(n^2+2(\alpha-1)-n\right)}
    {(n\rho+3)(n\rho+2)}+\dfrac{x\rho\left(n\rho^2+3n\rho-2(\alpha-1)\rho^2\right)}{(n\rho+3)(n\rho+2)}+
    \dfrac{2}{(n\rho+3)(n\rho+2)};$
\item $\mathcal{G}_{n,\rho}^{(\alpha)}(e_3;x)=\dfrac{x^3\rho^3\left(n^3+6n\alpha-3n^2-4n-12(\alpha-1)\right)}{(n\rho+4)(n\rho+3)(n\rho+2)}
    \\+\dfrac{3x^2\rho^2\left(6n^2+3n\rho+3n^2\rho-6n\alpha\rho-6n+6(\alpha-1)(2+3\rho)\right)}{(n\rho+4)(n\rho+3)(n\rho+2)}
+
   \dfrac{x\rho\left(n\rho^2+6n\rho+11n-6(\alpha-1)\rho(2+\rho)\right)}{(n\rho+4)(n\rho+3)(n\rho+2)}\\+\dfrac{6}{(n\rho+4)(n\rho+3)(n\rho+2)};$
    \item $\mathcal{G}_{n,\rho}^{(\alpha)}(e_4;x)=\dfrac{x^4\rho^4\left(n^4-6n^3+72(\alpha-1)-6n(10\alpha-9)+n^2(12\alpha-1)\right)}{(n\rho+5)(n\rho+4)(n\rho+3)(n\rho+2)}
       \\+ \dfrac{x^3\rho^3}{(n\rho+5)(n\rho+4)(n\rho+3)(n\rho+2)}\bigg[10n^3-30n^2+10n(6\alpha-4)-7n^2\rho+6n^3\rho+6n(6\alpha-5)\rho+6n(10\alpha-9)\rho
       +n^2(12\alpha-1)\rho-24(\alpha-1)(6\rho+5)\bigg]+\dfrac{x^2\rho^2}
       {(n\rho+5)(n\rho+4)(n\rho+3)(n\rho+2)}
       \bigg[35n(n-1)-10n\rho+30n^2\rho-10n(6\alpha-4)\rho
       -n\rho^2+7n^2\rho^2-6n(6\alpha-5)\rho^2+2(\alpha-1)
       (43\rho^2+90\rho+35)\bigg]\\+ \dfrac{x\rho\left(35n\rho+50n+10n\rho^2+n\rho^3-2(\alpha-1)\rho(7\rho^2+30\rho+35)\right)}{(n\rho+5)(n\rho+4)(n\rho+3)(n\rho+2)}+\dfrac{24}{(n\rho+5)(n\rho+4)(n\rho+3)(n\rho+2)}.$
 \end{enumerate}
\end{lemma}

\begin{lemma}\label{l2}
 For $m=1,2,$ the $m^{th}$ order central moments of $ \mathcal{G}_{n,\rho}^{(\alpha)}$  defined as $\tau_{n,\rho,m}^{(\alpha)}(x)=\mathcal{G}_{n,\rho}^{(\alpha)}((t-x)^m;x)$  we get
\begin{enumerate}[(i)]
\item $\tau_{n,\rho,1}^{(\alpha)}(x)=\dfrac{1-2x}{(n\rho+2)};$
\item $\tau_{n,\rho,2}^{(\alpha)}(x)=\dfrac{x(1-x)\left(\rho(n+(n-2\alpha+2)\rho)-6\right)}
    {(n\rho+2)(n\rho+3)}+\dfrac{2}{(n\rho+2)(n\rho+3)}.$
 \end{enumerate}
\end{lemma}

\begin{remark}\label{r1} For every $x\in J,$ we have
\begin{eqnarray*}
\displaystyle \lim_{n\rightarrow\infty} n~\tau_{n,\rho,1}^{(\alpha)}(x)&=&\dfrac{1-2x}{\rho},\\
\displaystyle \lim_{n\rightarrow\infty}n ~\tau_{n,\rho,2}^{(\alpha)}(x)&=&\frac{x(1-x)(1+\rho)}{\rho},\\
\displaystyle \lim_{n\rightarrow\infty} n^2~\tau_{n,\rho,4}^{(\alpha)}(x)&=&\frac{3x^2(1-x)^2(1+\rho)^2}{\rho^2}.
\end{eqnarray*}
\end{remark}
\begin{lemma}\label{l3} For $n\in\mathbb{N}$, we obtain
$$\mathcal{G}_{n,\rho}^{(\alpha)}((t-x)^2;x)\leq\frac{\mathcal{X}_{\rho}^{(\alpha)}~~ x(1-x)}{(1+n\rho)},$$
where $\mathcal{X}_{\rho}^{(\alpha)}$ is a positive constant depending on $\alpha$ and $\rho.$
\end{lemma}

\section{Direct Estimates}

\begin{thm}\label{t1}
Let $f\in C(J).$ Then $\displaystyle\lim_{n\rightarrow\infty}\mathcal{G}_{n,\rho}^{(\alpha)}(f;x)=f(x),$ uniformly on $J$.
\end{thm}
\begin{proof}
 In view of Lemma \ref{l1}, $\mathcal{G}_{n,\rho}^{(\alpha)}(1;x)=1,$ $\mathcal{G}_{n,\rho}^{(\alpha)}(e_1;x)\rightarrow x,$ $\mathcal{G}_{n,\rho}^{(\alpha)}(e_2;x)\rightarrow x^2$ as $n\rightarrow\infty,$ uniformly in $J$.
 Applying Bohman-Korovkin criterion, it follows that $\mathcal{G}_{n,\rho}^{(\alpha)}(f;x)\rightarrow f(x)$ as $n\rightarrow\infty,$ uniformly on $J$.
\end{proof}

\subsection{Voronovskaja type theorem}
In this section we prove Voronvoskaja type theorem for the operators $\mathcal{G}_{n,\rho}^{(\alpha)}$.
\begin{thm}\label{tee1}
Let $f\in C(J).$ If $f''$  exists at a point $x\in J,$ then we have
$$\displaystyle\lim_{n\to \infty}n\left[\mathcal{G}_{n,\rho}^{(\alpha)}(f;x)-f(x)\right]=\displaystyle\dfrac{1-2x}{\rho}f'(x)+\frac{(1+\rho)x(1-x)}{2\rho}f^{\prime\prime}(x).  $$
\end{thm}
\begin{proof}
By Taylor's expansion of $f$, we get
\begin{eqnarray}\label{ee1}
f(t)=f(x)+f'(x)(t-x)+\frac{1}{2}f''(x)(t-x)^2+\varpi(t,x)(t-x)^2,
\end{eqnarray}
where $\displaystyle\lim_{t\rightarrow x}\varpi(t,x)=0$. By applying the linearity of the operator $\mathcal{G}_{n,\rho}^{(\alpha)}$, we obtain \
\begin{align*}
\mathcal{G}_{n,\rho}^{(\alpha)}(f;x)-f(x)&=\mathcal{G}_{n,\rho}^{(\alpha)}((t-x);x)f'(x)+\frac{1}{2}\mathcal{G}_{n,\rho}^{(\alpha)}((t-x)^2;x)f''(x)\\
&+\mathcal{G}_{n,\rho}^{(\alpha)}(\varpi(t,x)(t-x)^2;x).
\end{align*}
Now, applying Cauchy-Schwarz property, we can get
$$n\mathcal{G}_{n,\rho}^{(\alpha)}(\varpi(t,x)(t-x)^2;x)\leq\sqrt{\mathcal{G}_{n,\rho}^{(\alpha)}(\varpi^2(t,x);x)}
\sqrt{n^2\mathcal{G}_{n,\rho}^{(\alpha)}((t-x)^4;x)}.$$
From Theorem \ref{t1}, we have $\displaystyle\lim_{n\rightarrow\infty}\mathcal{G}_{n,\rho}^{(\alpha)}(\varpi^2(t,x);x)$= $\varpi^2(x,x)=0,$ since $\varpi(t,x)\rightarrow0$ as $t\rightarrow x,$ and Remark \ref{r1} for every $x\in J,$ we may write
\begin{align}\label{E1}\displaystyle\lim_{n\to \infty}n^2\mathcal{G}_{n,\rho}^{(\alpha)}\left((t-x)^4;x\right)&=\frac{3x^2(1-x)^2(1+\rho)^2}{\rho^2}. \end{align}
Hence,
\begin{eqnarray*}
n\mathcal{G}_{n,\rho}^{(\alpha)}(\varpi(t,x)(t-x)^2;x)=0.
\end{eqnarray*}
Applying Remark \ref{r1}, we get
\begin{align}\label{E2}
& \displaystyle \lim_{n\to \infty}n\mathcal{G}_{n,\rho}^{(\alpha)}\left(t-x;x\right)=\dfrac{1-2x}{\rho},\nonumber\\
&\displaystyle\lim_{n\to\infty}n\mathcal{G}_{n,\rho}^{(\alpha)}\left((t-x)^2;x\right)=\frac{(1+\rho)x(1-x)}{\rho}.
\end{align}
Collecting the results from above the theorem is completed.
\end{proof}

\subsection{Local approximation}
We begin by recalling the following K-functional :\\
$$K_2(f,\delta)=\inf\{|| f-g||+\delta||g''||:g\in W^2\}\,\,(\delta>0),$$\\
where $W^2=\{g:g''\in C(J)\}$ and $||.||$ is the uniform norm on $C(J).$ By \cite{DL}, $\exists$ a positive constant $M>0$ such that
\begin{eqnarray}\label{kk1}
K_2(f,\delta)\leq M\omega_2(f,\sqrt{\delta}),\
\end{eqnarray}
where the  modulus of smoothness of second order for $f\in C(J)$ is defined as
$$\omega_2(f,\sqrt{\delta})=\displaystyle\sup_{0<h\leq\sqrt{\delta}}\displaystyle\sup_{x,x+2h\in J}|f(x+2h)-2f(x+h)+f(x)|.$$
 The modulus of continuity for $f\in C(J)$ is defined by
$$\omega(f,\delta)=\displaystyle\sup_{0<h\leq\delta}\sup_{x,x+h\in J}|f(x+h)-f(x)|.$$\\
The Steklov mean is defined as
\begin{eqnarray}\label{ss1}
f_h(x)=\displaystyle\dfrac{4}{h^2}\int_0^{\frac{h}{2}}\int_0^{\frac{h}{2}}
\left[2f(x+u+v)-f(x+2(u+v))\right]du~dv.
\end{eqnarray}
 The Steklov mean satisfies the following inequality:
 \begin{itemize}
 \item[a)] $\Vert f_h-f\Vert_{C(J)}\leq\omega_2(f,h).$
 \item[b)]  $f'_h,f_h''\in C(J)$ and $\Vert f'_h\Vert_{C(J)}\leq\dfrac{5}{h}\omega(f,h),\quad \Vert f''_h\Vert_{C(J)}\leq\dfrac{9}{h^2}\omega_2(f,h)$,
 \end{itemize}

\begin{thm}\label{t2}
For any function $f\in C(J)$. Then for each $x\in J,$ we have
\begin{eqnarray*}
\left|\mathcal{G}_{n,\rho}^{(\alpha)}(f;x)-f(x)\right|\leq 5  \omega\left(f,\sqrt{\tau_{n,\rho,2}^{(\alpha)}(x)}\right)
+\frac{13}{2}  \omega_2\left(f,\sqrt{\tau_{n,\rho,2}^{(\alpha)}(x)}\right).
\end{eqnarray*}
\end{thm}

\begin{proof}
For $x\in J,$ and applying the Steklov mean $f_h$ that is given by (\ref{ss1}), we can write
\begin{eqnarray}\label{ss2}
\left|\mathcal{G}_{n,\rho}^{(\alpha)}(f;x)\!-\!f(x)\right|\!\leq\! \mathcal{G}_{n,\rho}^{(\alpha)}\left(|f\!-\!f_h|;x\right)
\!+\!|\mathcal{G}_{n,\rho}^{(\alpha)}\left(f_h\!-\!f_h(x);x\right)|\!+\!|f_h(x)\!-\!f(x)|.
\end{eqnarray}
 From (\ref{ma2}), for each $f\in C(J)$ we obtain
\begin{eqnarray}\label{ch1}
\left|\mathcal{G}_{n,\rho}^{(\alpha)}(f;x)\right|\leq||f||.
\end{eqnarray}
By assumption $(a)$ of the Steklov mean and (\ref{ch1}), we get
$$\mathcal{G}_{n,\rho}^{(\alpha)}\left(|f-f_h|;x\right)
\leq\Vert\mathcal{G}_{n,\rho}^{(\alpha)}\left(f-f_h\right)\Vert
\leq\Vert f-f_h\Vert\leq\omega_2(f,h).$$
Applying Taylor's expansion and Cauchy-Schwarz inequality, we have
\begin{align*}\left|\mathcal{G}_{n,\rho}^{(\alpha)}\left(f_h-f_h(x);x\right)\right|&\leq\Vert f_h'\Vert\sqrt{\mathcal{G}_{n,\rho}^{(\alpha)}\left((t-x)^2;x\right)}+\dfrac{1}{2}\Vert f_h''\Vert\mathcal{G}_{n,\rho}^{(\alpha)}\left((t-x)^2;x\right).\end{align*}
 By Lemma \ref{l2} and property $(b)$ of the Steklov mean, we get
\begin{eqnarray*}
\left|\mathcal{G}_{n,\rho}^{(\alpha)}\left(f_h-f_h(x);x\right)\right|\leq\dfrac{5}{h}\omega(f,h)
\sqrt{\tau_{n,\rho,2}^{(\alpha)}(x)}+\dfrac{9}{2h^2}\omega_2(f,h)\tau_{n,\rho,2}^{(\alpha)}(x).
\end{eqnarray*}
Finally, choosing $h=\sqrt{\tau_{n,\rho,2}^{(\alpha)}(x)}$, we
obtain the desired result.
\end{proof}

\subsection{Global Approximation}
 Now, we recall the definitions of the Ditzian-Totik first order modulus of continuity and the $K$-functional \cite{ZD}. Let $
\phi(x) =\sqrt{x(1-x) }$ and $f\in C(J).$ The first order modulus of smoothness is defined by
\begin{eqnarray*}
\omega _{\phi }(f,t)=\sup_{0<h\leq t}\bigg\{\left\vert f\bigg(x+\frac{h\phi(x)}{2}\bigg)-f\bigg(x-\frac{h\phi(x)}{2}\bigg)\right\vert ,x\pm\frac{h\phi(x)}{2}\in J\bigg\} ,
\end{eqnarray*}
and the  corresponding $K$-functional is given by
\begin{eqnarray*}
\overline{K}_{\phi }(f,t)=\inf_{g\in W_{\phi}}\{||f-g||+t||\phi g^{\prime}|| +t^{2}|| g^{\prime }||\}\,\,(t>0),
\end{eqnarray*}

where $W_{\phi }=\{g:g\in AC_{loc},||\phi g^{\prime}||<\infty,||g^{\prime}||<\infty \}$ and $||.||$ is the uniform norm on $C(J).$
It is well known that  (Theorem 3.1.2, \cite{ZD}) $\overline{K}_{\phi }(f,t)\sim \omega _{\phi }(f,t)$ which means that there exists a constant $M>0$ such that
\begin{eqnarray}\label{e7}
M^{-1}\omega _{\phi }(f,t)\leq\overline{K}_{\phi}(f,t)\leq M\omega _{\phi }(f,t).
\end{eqnarray}
Now, we establish the order of approximation with the aid of the Ditzian-Totik modulus of the first and second order.
\begin{thm}\label{t11}
Let $f$ be in $C(J) $ and $\phi(x) =\sqrt{x(1-x)},$ then for each $x\in[0,1),$ we get
\begin{eqnarray*}
|\mathcal{G}_{n,\rho}^{(\alpha)}(f;x)-f(x)|&\leq& C \omega_{\phi}\left(f,\sqrt{\frac{\mathcal{X}_{\rho}^{(\alpha)}}{(1+n\rho)}}\right),
\end{eqnarray*}
where $\mathcal{X}_{\rho}^{(\alpha)}$ is defined in Lemma \ref{l3} and $C>0$ is a constant.
\end{thm}
\begin{proof}
 By using the relation $\displaystyle g(t)=g(x)+\int_{x}^{t}g^{\prime}(u)du,$ we can write
\begin{eqnarray}\label{ki1}
\left\vert \mathcal{G}_{n,\rho}^{(\alpha)}(g;x)-g(x)\right\vert&=&\left\vert \mathcal{G}_{n,\rho}^{(\alpha)}\bigg(\int_{x}^{t}g^{\prime}(u)du;x\bigg)\right\vert.
\end{eqnarray}
For any $x,t\in(0,1),$ we get
\begin{eqnarray}\label{ki2}
\bigg|\int_{x}^{t}g^{\prime}(u)du\bigg|\leq||\phi g'||\bigg|\int_{x}^t\frac{1}{\phi(u)}du\bigg|.
\end{eqnarray}
Therefore,
\begin{eqnarray}\label{ki3}
\bigg|\int_{x}^t\frac{1}{\phi(u)}du\bigg|&=&\bigg|\int_{x}^t\frac{1}{\sqrt{u(1-u)}}du\bigg|
\leq\bigg|\int_{x}^t\bigg(\frac{1}{\sqrt{u}}+\frac{1}{\sqrt{1-u}}\bigg)du\bigg|\nonumber\\
&\leq&2\bigg(\mid\sqrt{t}-\sqrt{x}\mid+\mid\sqrt{1-t}-\sqrt{1-x}\mid\bigg)\nonumber\\
&=&2|t-x|\bigg(\frac{1}{\sqrt{t}+\sqrt{x}}+\frac{1}{\sqrt{1-t}+\sqrt{1-x}}\bigg)\nonumber\\
&<&2|t-x|\bigg(\frac{1}{\sqrt{x}}+\frac{1}{\sqrt{1-x}}\bigg)
\leq\frac{2\sqrt{2}~|t-x|}{\phi(x)}.
\end{eqnarray}
 Combining (\ref{ki1})-(\ref{ki3}) and applying Cauchy-Schwarz inequality, we have
\begin{eqnarray*}
|\mathcal{G}_{n,\rho}^{(\alpha)}(g;x)-g(x)| &<& 2\sqrt{2}||\phi g'||\phi^{-1}(x)\mathcal{G}_{n,\rho}^{(\alpha)}(|t-x|;x)\\
&\leq&2\sqrt{2}||\phi g'||\phi^{-1}(x)\bigg( \mathcal{G}_{n,\rho}^{(\alpha)}((t-x)^2;x)\bigg)^{1/2}.
\end{eqnarray*}
From Lemma \ref{l3}, we get
\begin{eqnarray}\label{k1}
|\mathcal{G}_{n,\rho}^{(\alpha)}(g;x)-g(x)|< C\sqrt{\frac{\mathcal{X}_{\rho}^{(\alpha)}}{(1+n\rho)}}||\phi g'||.
\end{eqnarray}
Applying Lemma \ref{l1} and (\ref{k1}), we get
\begin{eqnarray}\label{kk1}
\mid\mathcal{G}_{n,\rho}^{(\alpha)}(f)-f\mid &\leq&\mid\mathcal{G}_{n,\rho}^{(\alpha)}(f-g;x)\mid +|f-g|+\mid \mathcal{G}_{n,\rho}^{(\alpha)}(g;x)-g(x)\mid\nonumber \\
&\leq &C\left(||f-g||+\sqrt{\frac{\mathcal{X}_{\rho}^{(\alpha)}}{(1+n\rho)}}||\phi g'||\right).
\end{eqnarray}

Taking infimum on the right hand side of the above property over all $g\in W_\phi,$ we may write
\begin{eqnarray}\label{k3}
|\mathcal{G}_{n,\rho}^{(\alpha)}(f;x)-f(x)|\leq C \overline{K}_{\phi}\left(f;\sqrt{\frac{\mathcal{X}_{\rho}^{(\alpha)}}{(1+n\rho)}}\right).
\end{eqnarray}
Using $\overline{K_\phi}(f,t)\sim\omega_\phi(f,t)$, we immediately arrive to the required relation.
\end{proof}
 \cite{OA} Let us consider the Lipschitz-type space with two parameters $\kappa_1\geq0, \kappa_2>0,$  we have
\begin{eqnarray*}
Lip_M^{(\kappa_1,\kappa_2)}(\sigma):=\left\{f\in C(J):|f(t)-f(x)|\leq M\frac{|t-x|^{\sigma}}{(t+\kappa_1x^2+\kappa_2x)^{\frac{\sigma}{2}}};t\in J, x\in(0,1]\right\},
\end{eqnarray*}
where $0<\sigma\leq1.$
\begin{thm}
Let $f\in Lip_M^{(\kappa_1,\kappa_2)}(\sigma)$. Then for all $x\in(0,1],$ we have
\begin{eqnarray*}
\left|\mathcal{G}_{n,\rho}^{(\alpha)}(f;x)-f(x)\right|\leq M\left(\frac{\tau_{n,\rho,2}^{(\alpha)}(x)}{\kappa_1x^2+\kappa_2x}\right)^{\sigma/2}.
\end{eqnarray*}
\end{thm}
\begin{proof}
Let we show the theorem for the case $0<\sigma\leq1$, using  Holder's property with $p=\frac{2}{\sigma}, q=\frac{2}{2-\sigma}.$
\begin{eqnarray*}
\left|\mathcal{G}_{n,\rho}^{(\alpha)}(f;x)-f(x)\right|&\leq&\displaystyle\sum_{k=0}^{n}p_{n,k}^{(\alpha)}(x)\displaystyle\int_0^1\left|f(t)-f(x)\right|\mu_{n,\rho}(t)dt\\
&\leq&\displaystyle\sum_{k=0}^{n}p_{n,k}^{(\alpha)}(x)\left(\displaystyle\int_0^1\left|f(t)-f(x)\right|^{\frac{2}{\sigma}}\mu_{n,\rho}(t)dt\right)^{\frac{\sigma}{2}}\\
&\leq&\left\{\displaystyle\sum_{k=0}^{n}p_{n,k}^{(\alpha)}(x)\displaystyle\int_0^1\left|f(t)-f(x)\right|^{\frac{2}{\sigma}}\mu_{n,\rho}(t)dt\right\}^{\frac{\sigma}{2}}\\&&\times
\left(\displaystyle\sum_{k=0}^{n}p_{n,k}^{(\alpha)}(x)\int_0^1\mu_{n,\rho}(t)dt\right)^{\frac{2-\sigma}{2}}\\
&=&\left(\displaystyle\sum_{k=0}^{n}p_{n,k}^{(\alpha)}(x)\displaystyle\int_0^1\left|f(t)-f(x)\right|^{\frac{2}{\sigma}}\mu_{n,\rho}(t)dt\right)^{\frac{\sigma}{2}}\\
&\leq& M \left(\displaystyle\sum_{k=0}^{n}p_{n,k}^{(\alpha)}(x)\displaystyle\int_0^1\frac{(t-x)^2}{(t+\kappa_1x^2+\kappa_2x)}\mu_{n,\rho}(t)dt\right)^{\frac{\sigma}{2}}\\
&\leq& \frac{M}{\left(\kappa_1x^2+\kappa_2x\right)^{\frac{\sigma}{2}}}\left(\displaystyle\sum_{k=0}^{n}p_{n,k}^{(\alpha)}(x)\displaystyle\int_0^1(t-x)^2\mu_{n,\rho}(t)dt\right)^{\frac{\sigma}{2}}\\
&=&\frac{M}{\left(\kappa_1x^2+\kappa_2x\right)^{\frac{\sigma}{2}}}\mathcal{G}_{n,\rho}^{(\alpha)}((t-x)^2;x)^{\frac{\sigma}{2}}\\
&=&\frac{M}{\left(\kappa_1x^2+\kappa_2x\right)^{\frac{\sigma}{2}}}(\tau_{n,\rho,2}^{(\alpha)}(x))^{\frac{\sigma}{2}}.
\end{eqnarray*}
\end{proof}

\begin{thm}
 For$f\in C^1(J)$ and $x\in J,$ we have
\begin{eqnarray}\label{p1}
\left|\mathcal{G}_{n,\rho}^{(\alpha)}(f;x)-f(x)\right|\leq\left|\dfrac{1-2x}{(n\rho+2)}\right||f'(x)|+2\sqrt{\tau_{n,\rho,2}^{(\alpha)}(x)}\, \omega\left(f',\sqrt{\tau_{n,\rho,2}^{(\alpha)}(x)}\right). \end{eqnarray}
\end{thm}
\begin{proof}
Let $f\in C^1(J)$. For any $t,x\in J,$ we have
$$f(t)-f(x)=f'(x)(t-x)+\int_{x}^t\left(f'(u)-f'(x)\right)du.$$
Using $\mathcal{G}_{n,\rho}^{(\alpha)}(\cdot;x)$ on both sides of the above relation, we may write
$$\mathcal{G}_{n,\rho}^{(\alpha)}(f(t)-f(x);q_n,x)=f'(x)\mathcal{G}_{n,\rho}^{(\alpha)}(t-x;x)+
\mathcal{G}_{n,\rho}^{(\alpha)}\left(\int_{x}^t\left(f'(u)-f'(x)\right)du;x\right)$$
Using the well-known inequality of modulus of continuity $|f(t)-f(x)|\leq\omega(f,\delta)\left(\frac{|t-x|}{\delta}+1\right),\delta>0,$ we obtain
$$\left|\int_{x}^t\left(f'(u)-f'(x)\right)du\right|\leq\omega(f',\delta)\left(\frac{(t-x)^2}{\delta}+|t-x|\right),$$
it follows that
$$\left|\mathcal{G}_{n,\rho}^{(\alpha)}(f;x)-f(x)\right|\leq|f'(x)|~~~|\mathcal{G}_{n,\rho}^{(\alpha)}(t-x;x)|+
\omega(f',\delta)\left\{\frac{1}{\delta}\mathcal{G}_{n,\rho}^{(\alpha)}((t-x)^2;x)+\mathcal{G}_{n,\rho}^{(\alpha)}(|t-x|;x)\right\}.$$
From Cauchy-Schwarz inequality, we have
\begin{eqnarray*}\left|\mathcal{G}_{n,\rho}^{(\alpha)}(f;x)-f(x)\right|&\leq&|f'(x)|~~~|\mathcal{G}_{n,\rho}^{(\alpha)}(t-x;x)|\\&&+
\omega(f',\delta)\left\{\dfrac{1}{\delta}\sqrt{\mathcal{G}_{n,\rho}^{(\alpha)}((t-x)^2;x)}+1\right\}
\sqrt{\mathcal{G}_{n,\rho}^{(\alpha)}((t-x)^2;x)}.
\end{eqnarray*}
Now, choosing $\delta=\sqrt{\tau_{n,\rho,2}^{(\alpha)}(x)},$ the required result follows.
\end{proof}

\subsection{Rate of convergence}
Let $DBV_{(J)}$ be the class of all absolutely continuous functions $f$ defined on $J$, having on $J$ a derivative $f^{\prime}$ equivalent with a function of bounded variation on $J$. We observed that the functions $f$ $\in DBV_{(J)}$ possess a representation
\begin{eqnarray*}
f(x)=\int_{0}^xg(t)dt+f(0)
\end{eqnarray*}
where $g\in BV_{(J)}$, i.e., $g$ is a function of bounded variation on $J$.

The operators $\mathcal{G}_{n,\rho}^{(\alpha)}(f;x)$ also admit the integral representation
\begin{eqnarray}\label{q4}
\mathcal{G}_{n,\rho}^{(\alpha)}(f;x)=\int_{0}^{1}\mathcal{U}_{n,\rho}^{(\alpha)}(x,t) f(t) dt,
\end{eqnarray}
where the kernel $\mathcal{U}_{n,\rho}^{(\alpha)}(x,t)$ is given by
$$\mathcal{U}_{n,\rho}^{(\alpha)}(x,t)=\sum_{k=0}^{n}p_{n,k}^{(\alpha)}(x)\mu_{n,\rho}(t).$$

\begin{lemma}\label{lma78}
For a fixed $x\in(0,1)$ and sufficiently large $n$, we have
\begin{enumerate}[(i)]
\item $\gamma_{n,\rho}^{(\alpha)}(x,y)=\displaystyle\int_{0}^{y}\mathcal{U}_{n,\rho}^{(\alpha)}(x,t)dt\leq\dfrac{\mathcal{X}_{\rho}^{(\alpha)}}{(1+n\rho)}\frac{x(1-x)}{(x-y)^{2}}\,,\,\,0\leq y<x,$
\item $1-\gamma_{n,\rho}^{(\alpha)}(x,z)=\displaystyle\int_{z}^{1}\mathcal{U}_{n,\rho}^{(\alpha)}(x,t)dt\leq $ $\dfrac{\mathcal{X}_{\rho}^{(\alpha)}}{(1+n\rho)}\dfrac{x(1-x)}{(z-x)^{2}},$ $x<z<1,$
\end{enumerate}
where $\mathcal{X}_{\rho}^{(\alpha)}$ is defined in Lemma \ref{l3}.
\end{lemma}
\begin{proof}
$(i)$ From Lemma \ref{l3}, we get
\begin{eqnarray*}
\gamma_{n,\rho}^{(\alpha)}(x,y) &=&\int\limits_{0}^{y}\mathcal{U}_{n,\rho}^{(\alpha)}(x,t)dt\leq\int_{0}^{y}\bigg(\frac{x-t}{x-y}\bigg)^{2}\mathcal{U}_{n,\rho}^{(\alpha)}(x,t)dt \\
&=&\mathcal{G}_{n,\rho}^{(\alpha)}((t-x)^{2};x)( x-y) ^{-2}\leq \dfrac{\mathcal{X}_{\rho}^{(\alpha)}}{(1+n\rho)}\dfrac{x(1-x)}{(x-y)^{2}}.
\end{eqnarray*}
The proof of $(ii)$ is similar hence the details are missing.
\end{proof}

\begin{thm}
Let $f$ $\in DBV(J).$ Then for every $x\in(0,1)$ and sufficiently large $n$, we have
\begin{eqnarray*}
| \mathcal{G}_{n,\rho}^{(\alpha)}(f;x)-f(x)|  &\leq &\frac{(1-2x)}{(n\rho+2)}\frac{| f^{\prime }(x+)+f^{\prime }(x-)|}{2}+\sqrt{\frac{\mathcal{X}_{\rho}^{(\alpha)} x(1-x)}{(1+n\rho)}}\frac{|f^{\prime }(x+)-f^{\prime}(x-)|}{2}
\\&&+\frac{\mathcal{X}_{\rho}^{(\alpha)} (1-x)}{(1+n\rho)}\sum_{k=1}^{[\sqrt{n}] }\bigvee_{x-(x/k)}^{x}(f^{\prime}_{x}) +
\frac{x}{\sqrt{n}}\bigvee_{x-(x/\sqrt{n})}^{x}(f^{\prime}_x)
\\&&+\frac{\mathcal{X}_{\rho}^{(\alpha)} x}{(1+n\rho)}\sum_{k=1}^{[\sqrt{n}] }\bigvee_{x}^{x+((1-x)/k)}(f^{\prime}_x) +\frac{(1-x)}{\sqrt{n}}\bigvee_{x}^{x+((1-x)/\sqrt{n})}(f^{\prime}_x),
\end{eqnarray*}
where $\bigvee_{c}^{d}(f^{\prime}_{x}) $ denotes the total variation of $f^{\prime}_{x} $ on $[c,d] $ and $f^{\prime}_{x}$ is defined by
\begin{eqnarray}\label{q5}
f_{x}^{\prime}(t)=\left\{
\begin{array}{cc}
f^{\prime}(t)-f^{\prime}(x-), & 0\leq t<x \\
0, & t=x \\
f^{\prime}(t)-f^{\prime}(x+) & x<t<1.
\end{array}
\right.
\end{eqnarray}
\end{thm}

\begin{proof}
Since $\mathcal{G}_{n,\rho}^{(\alpha)}(1;x)=1,$ by using (\ref{q4}), for every $x\in(0,1)$ we may write
\begin{eqnarray}\label{q6}
\mathcal{G}_{n,\rho}^{(\alpha)}( f;x) -f(x)&=&\int_{0}^{1}\mathcal{U}_{n,\rho}^{(\alpha)}(x,t)(f(t)-f(x)) dt  \nonumber \\
&=&\int_{0}^{1}\mathcal{U}_{n,\rho}^{(\alpha)}(x,t)\left(\int_{x}^{t}f^{\prime}(u) du\right) dt.
\end{eqnarray}
For any $f\in DBV(J),$ by (\ref{q5}) we can write
\begin{eqnarray}\label{q7}
f^{\prime}(u) &=&f'_{x}(u)+\frac{1}{2}(
f^{\prime }(x+)+f^{\prime }(x-)) +\frac{1}{2}( f^{\prime}(x+)-f^{\prime }(x-)) sgn(u-x)  \nonumber \\
&&+\delta _{x}(u)[ f^{\prime }(u)-\frac{1}{2}( f^{\prime}(x+)+f^{\prime }(x-))],
\end{eqnarray}
where
\[
\delta _{x}(u)=\left\{
\begin{array}{c}
1\,\,,\,\,u=x \\
0\,\,,\,\,u\neq x.
\end{array}
\right.
\]
Obviously,
\begin{eqnarray*}
\int_{0}^{1}\bigg( \int_{x}^{t}\bigg(f^{\prime }(u)-\frac{1}{2}(f^{\prime }(x+)+f^{\prime }(x-))\bigg)\delta_{x}(u)du\bigg) \mathcal{U}_{n,\rho}^{(\alpha)}(x,t)dt=0.
\end{eqnarray*}
By (\ref{q4}) and simple calculations we find
\begin{eqnarray*}
\int_{0}^{1}\bigg( \int_{x}^{t}\frac{1}{2}( f^{\prime}(x+)+f^{\prime }(x-)) du\bigg) \mathcal{U}_{n,\rho}^{(\alpha)}(x,t)dt
&=&\frac{1}{2}(f^{\prime }(x+)+f^{\prime }(x-))\int_{0}^{1}( t-x) \mathcal{U}_{n,\rho}^{(\alpha)}(x,t)dt \\
&=&\frac{1}{2}( f^{\prime }(x+)+f^{\prime }(x-)) \mathcal{G}_{n,\rho}^{(\alpha)}( (t-x);x)
\end{eqnarray*}
and\\
\noindent
$\displaystyle\left|\int_{0}^{1}\mathcal{U}_{n,\rho}^{(\alpha)}(x,t)\bigg( \int_{x}^{t}\frac{1}{2}(f^{\prime }(x+)-f^{\prime }(x-)) sgn(u-x)du\bigg) dt\right|$
\begin{eqnarray*}
&\leq &\frac{1}{2}\mid f^{\prime }(x+)-f^{\prime }(x-)\mid\int_{0}^{1}|t-x|\mathcal{U}_{n,\rho}^{(\alpha)}(x,t)dt \\
&\leq &\frac{1}{2}\mid f^{\prime }(x+)-f^{\prime }(x-)\mid \mathcal{G}_{n,\rho}^{(\alpha)}(|t-x|;x)  \\
&\leq &\frac{1}{2}\mid f^{\prime }(x+)-f^{\prime }(x-)\mid\bigg(\mathcal{G}_{n,\rho}^{(\alpha)}((t-x) ^{2};x)\bigg) ^{1/2}.
\end{eqnarray*}
 By Lemma \ref{l2} and \ref{l3}, using (\ref{q6})-(\ref{q7}) we find
\begin{eqnarray}\label{q8}
| \mathcal{G}_{n,\rho}^{(\alpha)}(f;x)-f(x)|  &\leq &\frac{1}{2}|f^{\prime }(x+)-f^{\prime }(x-)| \sqrt{\frac{\mathcal{X}_{\rho}^{(\alpha)} x(1-x)}{(1+n\rho)}}\nonumber  \\
&&+\bigg|\int_{0}^{x}\left(\int_{x}^{t}f^{\prime}_{x}(u)du\right)\mathcal{U}_{n,\rho}^{(\alpha)}(x,t)dt\nonumber  \\&&+\int_{x}^{1}\left(\int_{x}^{t}f^{\prime}_{x}(u)du\right) \mathcal{U}_{n,\rho}^{(\alpha)}(x,t)dt\bigg|.
\end{eqnarray}
Let
\begin{eqnarray*}
\mathcal{S}_{n,\rho}^{(\alpha)}(f'_x,x)=\int_{0}^{x}\left(\int_{x}^{t}f'_{x}(u)du\right) \mathcal{U}_{n,\rho}^{(\alpha)}(x,t)dt,\\
\mathcal{T}_{n,\rho}^{(\alpha)}(f'_x,x)=\int_{x}^{1}\left(\int_{x}^{t}f'_{x}(u)du\right) \mathcal{U}_{n,\rho}^{(\alpha)}(x,t)dt.
\end{eqnarray*}
To complete the proof, it is sufficient to determine the terms $\mathcal{S}_{n,\rho}^{(\alpha)}(f'_x,x)$ and $\mathcal{T}_{n,\rho}^{(\alpha)}(f'_x,x).$
Since $\int_{c}^{d}d_{t}\gamma_{n,\rho}^{(\alpha)}(x,t)\leq 1$ for all $[c,d]\subseteq J,$ applying the integration by parts and applying Lemma \ref{lma78} with $y=x-(x/\sqrt{n}),$ we have
\begin{eqnarray*}
|\mathcal{S}_{n,\rho}^{(\alpha)}(f'_x,x)| &=&\bigg|\int_{0}^{x}\left(\int_{x}^{t}f'_{x}(u)du\right) d_{t}\gamma_{n,\rho}^{(\alpha)}(x,t)\bigg|\\
&=&\bigg|\int_{0}^{x}\gamma_{n,\rho}^{(\alpha)}(x,t)f'_{x}(t)dt\bigg|\\
&\leq&\bigg(\int_{0}^{y}+\int_{y}^{x}\bigg) |f'_{x}(t)|\,\,|\gamma_{n,\rho}^{(\alpha)}(x,t)| dt\\
&\leq&\frac{\mathcal{X}_{\rho}^{(\alpha)} x(1-x)}{(1+n\rho)}\int_{0}^{y}\bigvee_{t}^{x}(f'_{x})(x-t)^{-2}dt+\int_{y}^{x}\bigvee_{t}^{x}(f'_{x}) dt\\
&\leq&\frac{\mathcal{X}_{\rho}^{(\alpha)} x(1-x)}{(1+n\rho)}\int_{0}^{x-(x/\sqrt{n})}\bigvee_{t}^{x}(f^{\prime}_{x})( x-t)^{-2}dt+\frac{x}{\sqrt{n}}\bigvee_{x-(x/\sqrt{n})}^{x}(f'_{x}).
\end{eqnarray*}
By the substitution of $u=x/(x-t),$ we have
\begin{eqnarray*}
\frac{\mathcal{X}_{\rho}^{(\alpha)} x(1-x)}{(1+n\rho)}\int_{0}^{x-(x/\sqrt{n})}(x-t)^{-2}\bigvee_{t}^{x}(f^{\prime}_{x}) dt &=&\frac{\mathcal{X}_{\rho}^{(\alpha)} (1-x)}{(1+n\rho)}\int_{1}^{\sqrt{n}}\bigvee_{x-(x/u)}^{x}(f^{\prime}_{x}) du \\
&\leq&\frac{\mathcal{X}_{\rho}^{(\alpha)} (1-x)}{(1+n\rho)}\sum_{k=1}^{[\sqrt{n}]}\int_{k}^{k+1}\bigvee_{x-(x/u)}^{x}(f^{\prime}_{x}) du \\
&\leq&\frac{\mathcal{X}_{\rho}^{(\alpha)} (1-x)}{(1+n\rho)}\sum_{k=1}^{[\sqrt{n}] }\bigvee_{x-(x/k)}^{x}(f'_{x}) .
\end{eqnarray*}
Thus,
\begin{eqnarray}\label{q9}
|\mathcal{S}_{n,\rho}^{(\alpha)}(f'_x,x)| \leq \frac{\mathcal{X}_{\rho}^{(\alpha)} (1-x)}{(1+n\rho)}\sum_{k=1}^{[ \sqrt{n}] }\bigvee_{x-(x/k)}^{x}(f^{\prime}_{x}) +\frac{x}{\sqrt{n}}
\bigvee_{x-(x/\sqrt{n})}^{x}(f^{\prime}_{x}).
\end{eqnarray}
Using the integration by parts and Lemma \ref{lma78} with $z=x+((1-x)/\sqrt{n}),$ we can write
\begin{eqnarray*}
|\mathcal{T}_{n,\rho}^{(\alpha)}(f^{\prime}_x,x)| &=&\bigg|\int_{x}^{1}\left(\int_{x}^{t}f'_{x}(u)du\right) \mathcal{U}_{n,\rho}^{(\alpha)}(x,t)dt\bigg|\nonumber\\
&=&\bigg|\int_{x}^{z}\left(\int_{x}^{t}f'_{x}(u)du\right) d_{t}(1-\gamma_{n,\rho}^{(\alpha)}(x,t))+\int_{z}^{1}\left(\int_{x}^{t}f'_{x}(u)du\right)d_{t}(1-\gamma_{n,\rho}^{(\alpha)}(x,t))\bigg|\nonumber\\
&=&\bigg|\bigg[\int_{x}^tf'_{x}(u)(1-\gamma_{n,\rho}^{(\alpha)}(x,t))du\bigg]_{x}^z-\int_{x}^z f'_{x}(t)(1-\gamma_{n,\rho}^{(\alpha)}(x,t))dt\nonumber\\
&&+\int_{z}^{1}\left(\int_{x}^{t}f'_{x}(u)du\right)d_{t}(1-\gamma_{n,\rho}^{(\alpha)}(x,t))\bigg|\nonumber\\
&=&\bigg|\int_{x}^zf'_{x}(u)du(1-\gamma_{n,\rho}^{(\alpha)}(x,z))-\int_{x}^z f'_{x}(t)(1-\gamma_{n,\rho}^{(\alpha)}(x,t))dt+\bigg[\int_{x}^t f'_{x}(u)du(1-\gamma_{n,\rho}^{(\alpha)}(x,t))\bigg]_{z}^1\nonumber\\&&-\int_{z}^1f'_{x}(t)(1-\gamma_{n,\rho}^{(\alpha)}(x,t))dt\bigg|\nonumber\\
&=&\bigg| \int_{x}^z f'_{x}(t)(1-\gamma_{n,\rho}^{(\alpha)}(x,t))dt+\int_{z}^1f'_{x}(t)(1-\gamma_{n,\rho}^{(\alpha)}(x,t))dt\bigg|\nonumber\\
&\leq&\frac{\mathcal{X}_{\rho}^{(\alpha)} x(1-x)}{(1+n\rho)}\int_{z}^1 \bigvee_{x}^t(f'_{x})(t-x)^{-2}dt+\int_{x}^z\bigvee_{x}^t(f'_{x})dt\nonumber\\
&=&\frac{\mathcal{X}_{\rho}^{(\alpha)} x(1-x)}{(1+n\rho)}\int_{x+((1-x)/\sqrt{n})}^1\bigvee_{x}^t(f'_{x})(t-x)^{-2}dt+\frac{(1-x)}{\sqrt{n}}\bigvee_{x}^{x+(( 1-x) /\sqrt{n})}(f'_{x}).\nonumber
\end{eqnarray*}
By the substitution of $v=(1-x)/(t-x),$ we have
\begin{eqnarray}\label{q10}
|\mathcal{T}_{n,\rho}^{(\alpha)}(f^{\prime}_x,x)|&\leq&\frac{\mathcal{X}_{\rho}^{(\alpha)} x(1-x)}{(1+n\rho)}\int_{1}^{\sqrt{n}}\bigvee_{x}^{x+((1-x)/v)}(f'_{x})(1-x)^{-1}dv+\frac{(1-x)}{\sqrt{n}}\bigvee_{x}^{x+(( 1-x) /\sqrt{n})}(f'_{x})\nonumber\\
&\leq&\frac{\mathcal{X}_{\rho}^{(\alpha)} x}{(1+n\rho)}\sum_{k=1}^{[\sqrt{n}]}\int_{k}^{k+1}\bigvee_{x}^{x+((1-x)/v)}(f'_{x})dv+\frac{(1-x)}{\sqrt{n}}\bigvee_{x}^{x+(( 1-x) /\sqrt{n})}(f'_{x})\nonumber\\
&= &\frac{\mathcal{X}_{\rho}^{(\alpha)} x}{(1+n\rho)}\sum_{k=1}^{[ \sqrt{n}] }\bigvee_{x}^{x+((1-x)/k)}(f'_{x}) +
\frac{(1-x)}{\sqrt{n}}\bigvee_{x}^{x+((1-x))/\sqrt{n}}(f'_{x}).
\end{eqnarray}
Combining (\ref{q8})-(\ref{q10}), we get the desired relation.
\end{proof}

\section{Numerical Examples.}
\begin{example}
In Figure 1, for $n=20, \alpha=0.3, \rho=4,$ the comparison of
convergence of $\mathcal{G}_{20,4}^{(0.3)}(f;x)$ (blue) and the Bernstein-Durrmeyer $D_n(f;x)$ \cite{BD} (red) operators to $f(x)= x^2\sin\left(2x/\pi\right)$(yellow) is illustrated. It is observed that the $\mathcal{G}_{20,4}^{(0.3)}(f;x)$ operators gives a better approximation to $f(x)$ than  Bernstein-Durrmeyer $D_n(f;x)$ for $n=20, \alpha=0.3, \rho=4.$
\end{example}

\begin{center}
$\begin{array}{cc}
  \includegraphics[width=.6\columnwidth]{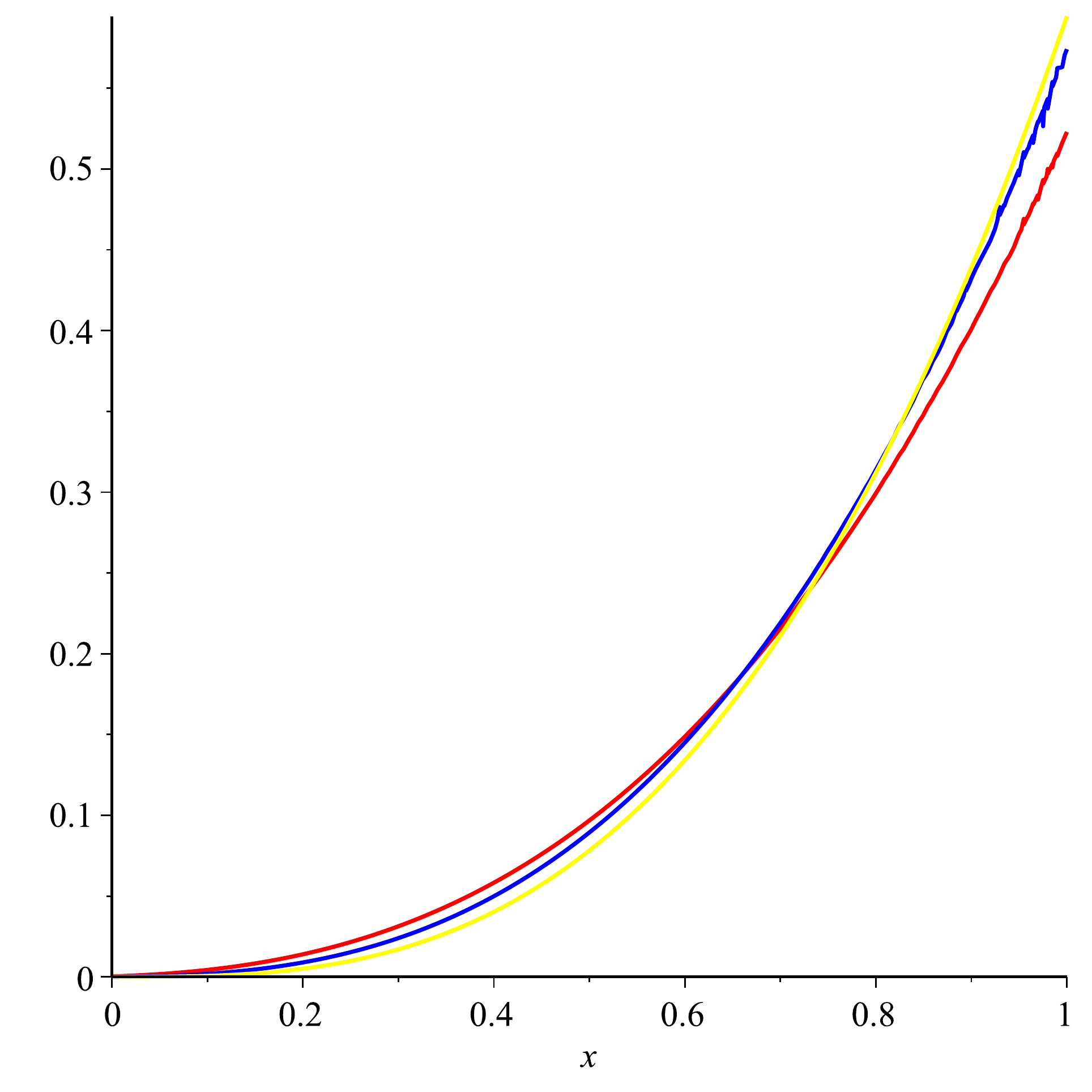} \\
     Figure ~~1. \mbox{The convergence of $\mathcal{G}_{20,4}^{(0.3)}(f;x)$ and $D_{20}(f;x)$ to $f(x)$}
     \end{array}$
  \end{center}

  \begin{example}
For $n\in\{10,20,50\}$, $\alpha=0.2$ and $\rho=4,$ the convergence of the operators $\mathcal{G}_{10,4}^{(0.2)}(f;x)$ (green), $\mathcal{G}_{20,4}^{(0.2)}(f;x)$ (red)  and $\mathcal{G}_{50,4}^{(0.2)}(f;x)$ (blue) to $f(x)=x^7+10x^5+x$ (yellow) is illustrated in Figure 2. We observed that for the values of $n$ increasing, the graph of $\mathcal{G}_{n,\rho}^{(\alpha)}(f;x)$ goes to the graph of
the function $f(x).$
\end{example}

\begin{center}
$\begin{array}{cc}
  \includegraphics[width=.6\columnwidth]{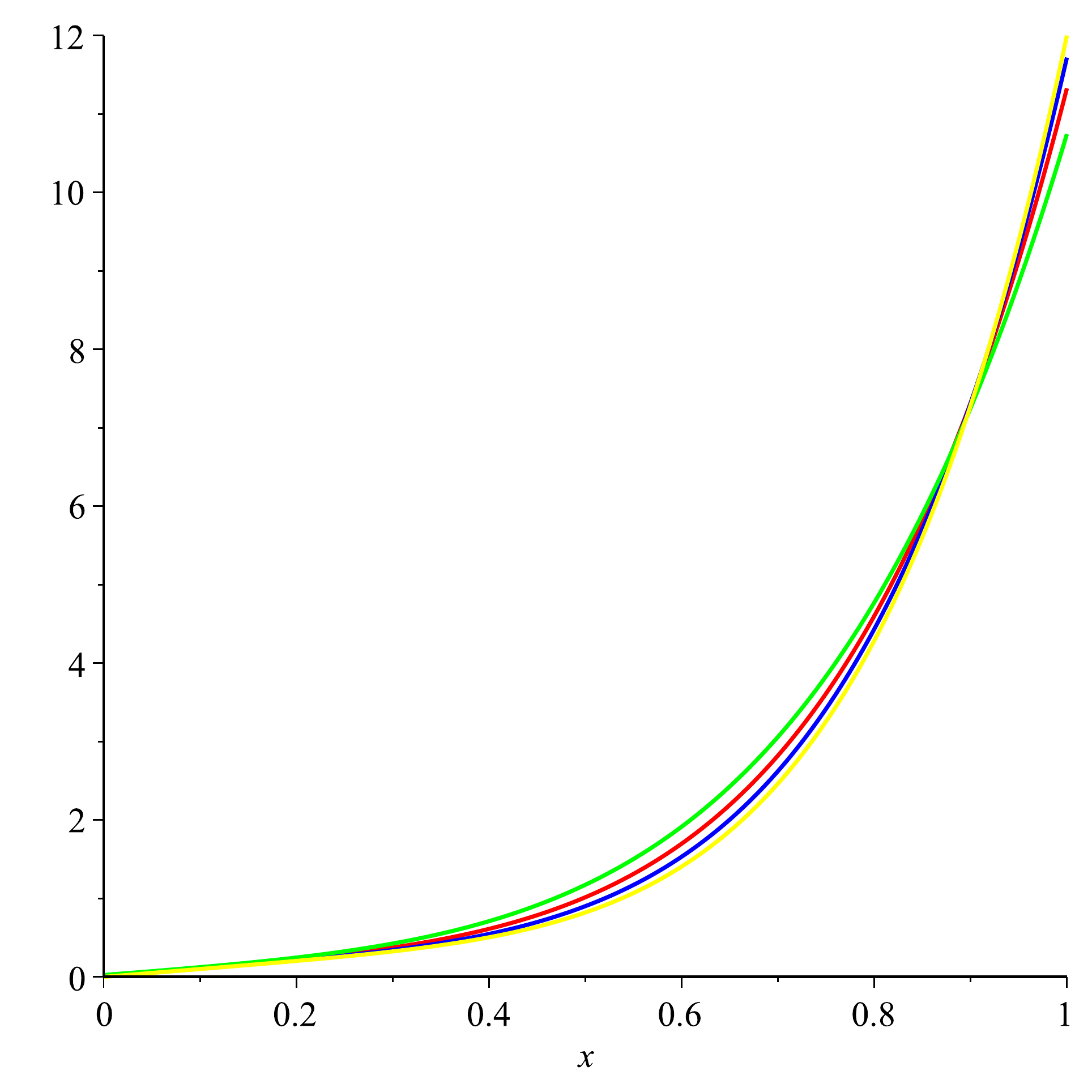} \\
     Figure ~~2. \mbox{The convergence of $\mathcal{G}_{n,\rho}^{(\alpha)}(f;x)$ to $f(x)$ }
   \end{array}$
  \end{center}

\end{document}